\documentclass{article}
\usepackage{amsmath}
\usepackage[noabbrev]{cleveref}
\theoremstyle{plain}
\newtheorem{thm}{Theorem}

\title{Isogenies over quadratic fields of elliptic curves with rational $j$-invariant}
\author{Borna Vukorepa}
\address{Department of Mathematics, University of Zagreb, Bijeni\v{c}ka cesta 30, 10000 Zagreb, Croatia}
\email{bovukor@math.hr}
\urladdr{https://web.math.pmf.unizg.hr/~bovukor/}

\begin{document}
\begin{abstract}
We determine the possible degrees of cyclic isogenies defined over quadratic fields for non-CM elliptic curves with rational $j$-invariant.
\end{abstract}
\maketitle
\section{Introduction}
For a number field $K$, it is well known that elliptic curves $E/K$ with a cyclic $K$-rational $n$-isogeny are parametrized (up to an isomorphism over $\overline{K}$) by the $K$-rational points on the modular curve $X_0(n)$. The curve $X_0(n)$ is, by definition, the completion of the curve $Y_0(n)$ isomorphic to $\Gamma_0(n)/\mathbb{H}$. Recall that:
\[
\mathbb{H} := \{x + yi : x,y \in \mathbb{R}, y > 0\}, \quad \Gamma_0(n) := \Bigg\{\begin{pmatrix}
a & b \\
c & d
\end{pmatrix}\ \in \SL_2(\mathbb{Z}) : c \equiv 0 \pmod{n}\Bigg\}
\]
All the $\mathbb{Q}$-rational points on $X_0(n)$ are known for all $n$, by the results of Mazur \cite{mazur77} and Kenku \cite{kenku1981} (see also references therein). With that information, we know that if $E/\mathbb{Q}$ is a non-CM elliptic curve with a cyclic $n$-isogeny defined over $\mathbb{Q}$, then $n \leq 18$ with $n \neq 14$ or $n \in \{21, 25, 37\}$. Recently, there has also been some progress in understanding the quadratic points on $X_0(n)$. Momose \cite[Theorem B]{momose} proved that for any fixed quadratic field $K$ which is not imagainary of class number $1$, $X_0(p)(K)$ has non-cuspidal points for only finitely many $p$. Bruin and Najman \cite{bnQcurve} described all quadratic points on the hyperelliptic curves $X_0(n)$ such that $J_0(n)(\Q)$ is finite (all hyperelliptic $X_0(n)$ except $X_0(37)$). Özman and Siksek \cite{OzmanSiksek19} described all quadratic points on the non-hyperelliptic modular curves $X_0(n)$ with $g(X_0(n))\leq 5$ and $\rk(J_0(n)(\Q))=0$, of which there are 15. After that, Box \cite{Box19} completed the description of quadratic points on $X_0(n)$ of genus $2\leq g(X_0(n)) \leq 5$ by describing the quadratic points for the 8 values of $n$ such that $\rk(J_0(n)(\Q))>0$, including the hyperelliptic case of $n=37$. Recently, Najman and the author of this paper \cite{NajmanVuk} have described all the quadratic points on bielliptic modular curves $X_0(n)$ for the values of $n$ not covered by the aforementioned papers.

All that information tells us something about the elliptic curves and isogenies defined over quadratic fields. It is therefore reasonable to ask ourselves a slightly different question: if $E$ has a rational $j$-invariant, what are the possible degrees of a cyclic isogeny of $E$ defined over a quadratic number field? This question is closely connected to the question of possible images of the mod $n$ Galois representation for an elliptic curve $E$ defined over $\mathbb{Q}$. There has been a lot of progress in that direction, one can, for example, see \cite{2AdicZB, reprZywina, MorrowSim, troglavaNeman, balDogra} and the references therein. It was proved by Najman in \cite{primeNajman} that if $E$ is a non-CM elliptic curve with a rational $j$-invariant and with an isogeny of prime degree $p$, then $p \in \{2,3,5,7,11,13,17,37\}$ as long as the degree of the field of definition of the isogeny is at most $7$. Motivated by that work, we will investigate the same question as investigated for quadratic number fields and for all possible cyclic isogeny degrees. Our main result is the following:
\begin{thm}\label{mainres}
Let $E$ be a non-CM elliptic curve with $j(E) \in \Q$. Assume $E$ has a cyclic $n$-isogeny defined over a quadratic extension of \space $\mathbb{Q}$. Then $n \leq 18$ with $n \neq 14$ or $n \in \{20, 21, 24, 25, 32, 36, 37\}$.
\end{thm}
The only new degrees not already arising for rational isogenies are $20, 24, 32, 36$ for which we can respectively find examples by simply using LMFDB: \href{http://www.lmfdb.org/EllipticCurve/2.0.8.1/2178.5/c/7}{2178.5-c7}, \href{http://www.lmfdb.org/EllipticCurve/2.2.40.1/90.1/f/3}{90.1-f3}, \href{http://www.lmfdb.org/EllipticCurve/2.2.5.1/45.1/a/3}{45.1-a3}, \href{http://www.lmfdb.org/EllipticCurve/2.0.7.1/28.2/a/11}{28.2-a11}. Performing these steps throughout the paper gives us the proof of Theorem \ref{mainres}:
\begin{enumerate}
    \item \hyperref[diffPrimes]{Show} that if $p < q$ are prime divisors of $n$, then $q \leq 5$ or $(p, q) \in \{(2, 7), (3, 7), (7 ,13)\}$.
    \item \hyperref[sqPrime]{Show} that if $p^2 \mid n$, then $p \in \{2,3,5\}$.
    \item \hyperref[5isogenies]{Show} that if $5^k \mid n$ or $3^k \mid n$, then $k \leq 2$.
    \item \hyperref[2pow]{Show} that if $2^k \mid n$, then $k \leq 5$.
    \item \hyperref[2and3]{Show} that if $n = 2^a3^b$, then $n \in \{2,4,6,8,9,12,16,18,24,32,36\}$.
    \item \hyperref[2and5]{Show} that if $n = 2^a5^b$, then $n \in \{2,4,5,8,10,16,20,25,32\}$.
    \item \hyperref[3and5]{Show} that if $n = 3^a5^b$, then $n \in \{3,5,9,15,25\}$.
    \item \hyperref[143063]{Show} that $n \in \{14,30,63\}$ is impossible.
    \item \hyperref[91iso]{Show} that $n = 91$ is impossible.
\end{enumerate}
The code that verifies all the computation in this paper can be found on:
\begin{center}
    \url{https://github.com/brutalni-vux/IsogeniesQuadratic}.
\end{center}

\section{Notation and auxiliary results}
Throughout this paper, $E/\mathbb{Q}$ will be a non-CM elliptic curve defined over $\mathbb{Q}$. For a number field $K$, denote with $G_K$ the absolute Galois group of $G$, meaning that $G_K = \Gal(\overline{K}/K)$. Notice that it is enough to consider elliptic curves defined over $\mathbb{Q}$. Presence of a cyclic $n$-isogeny is invariant under quadratic twisting. Hence, for $E/K$ with a $K$-rational cyclic $n$-isogeny we can descend from $E/K$ to $E'/\Q$ via a quadratic twist and isomorphism over $K$ with $E'$ having a $K$-rational cyclic $n$-isogeny. For a point $P \in E$, denote with $\langle P \rangle$ a cyclic subgroup of $E$ generated by $P$. For a finite cyclic subgroup $C$ of $E$, let $\mathbb{Q}(C)$ be the smallest field such that $G_{\mathbb{Q}(C)}$ acts on $C$. Notice that this is also the field of definition of an isogeny with kernel $C$. For a positive integer $n$, denote with $\rho_{E, n}: G_{\mathbb{Q}} \mapsto \GL_2(\mathbb{Z}/n\mathbb{Z})$ the mod $n$ Galois representation of $E$. In the remaining part of this section, we list several Lemmas and definitions which we will use frequently.

\begin{lemma}[{\cite[Proposition 3.1.]{primeNajman}}] \label{surjective}
Let $E/\mathbb{Q}$ be an elliptic curve and $p$ a prime such that $\rho_{E,p}$ is surjective, and $C$ a subgroup of E[p] of order p. Then $[\mathbb{Q}(C) : \mathbb{Q}] = p + 1$.
\end{lemma}

\begin{lemma}[{\cite[Lemma 3.2.]{primeNajman}}]
\label{interdeg}
Let $E/K$ be an elliptic curve over a number field and $P \in E[p]$. Let $C = \langle P \rangle$. Then
$[K(P) : K(C)]$ divides $p - 1$.
\end{lemma}

\begin{lemma}[{\cite[Proposition 3.3.]{primeNajman}}] \label{nonsplit}
Let $E/\mathbb{Q}$ be an elliptic curve and $p$ a prime such that the image of $\rho_{E,p}$ is
contained in the normalizer of the non-split Cartan subgroup and let $\langle P \rangle = C \subseteq E[p]$ a cyclic
subgroup of order $p$. Then:
\begin{itemize}
    \item If $p \equiv 1 \pmod{3}$, then $[\mathbb{Q}(C) : \mathbb{Q}] \geq p + 1$.
    \item If $p \equiv 2 \pmod {3}$, then $[\mathbb{Q}(C) : \mathbb{Q}] \geq (p + 1)/3$.
\end{itemize}
\end{lemma}

\begin{definition} \label{definedModP}
We say that the $p$-adic Galois representation $\rho_{E, p^{\infty}}: G_{\mathbb{Q}} \mapsto \GL_2(\mathbb{Z}_p)$ of $E$ is defined modulo $p^n$ if the
image $\rho_{E, p^{\infty}}(G_\mathbb{Q})$ contains the kernel of the reduction map $\GL_2(\mathbb{Z}_p) \mapsto \GL_2(\mathbb{Z}_p / p^n\mathbb{Z}_p)$.
\end{definition}

Notice that being defined modulo $p^n$ is invariant under conjugation.

\begin{lemma}[{\cite[Proposition 3.7.]{qOdd}}]
\label{definedP}
Let $E$ be an elliptic curve defined over a number field $K$ such that its $p$-adic representation is defined modulo $p^{n - 1}$ for some $n \geq 1$. Then for any cyclic subgroup $C$ of $E(\overline{K})$ of order $p^n$, we have
$[K(C) : K(pC)] = p$.
\end{lemma}

\section{Intermediate results}
\label{basicRes}
First we will draw several conclusions from some known results. Assume $E/\mathbb{Q}$ has a cyclic $n$-isogeny defined over a quadratic number field $K$ and $p \mid n$. Then $E$ has a $p$-isogeny defined over $K$. The main result of \cite{primeNajman} gives us that $p \in \{2,3,5,7,11,13,17,37\}$.

\subsection{Case $37 \mid n$} \label{data37}
We know from \cite[Theorem 1.11]{reprZywina} that $\rho_{E, 37}$ is surjective, conjugate to a subgroup of $N_{ns}(37)$, or $j(E) \in \{-7 \cdot 11^3, -7 \cdot 137^3 \cdot 2083^3\}$. Let $C$ be the kernel of a $37$-isogeny. If $\rho_{E, 37}$ is surjective or conjugate to a subgroup of $N_{ns}(37)$, we have from Lemmas \ref{surjective} and \ref{nonsplit} respectively that $[\mathbb{Q}(C) : \mathbb{Q}] \geq 38$. Otherwise, $j(E) \in \{-7 \cdot 11^3, -7 \cdot 137^3 \cdot 2083^3\}$ and there is a $37$-isogeny defined over $\mathbb{Q}$.

\subsection{Case $17 \mid n$} \label{data17}
We know from \cite[Theorem 1.11]{reprZywina} that $\rho_{E, 17}$ is surjective, conjugate to a subgroup of $N_{ns}(17)$, or $j(E) \in \{-17 \cdot 373^3 / 2^{17}, -17^2 \cdot 101^3 / 2\}$. Let $C$ be the kernel of a $17$-isogeny. If $\rho_{E, 17}$ is surjective or conjugate to a subgroup of $N_{ns}(17)$, we have from Lemmas \ref{surjective} and \ref{nonsplit} respectively that $[\mathbb{Q}(C) : \mathbb{Q}] \geq 6$. Otherwise, $j(E) \in \{-17 \cdot 373^3 / 2^{17}, -17^2 \cdot 101^3 / 2\}$ and there is a $17$-isogeny defined over $\mathbb{Q}$.

\subsection{Case $13 \mid n$} \label{data13}
We can use \cite[Theorem 1.8]{reprZywina} which tells us that the image of $\rho_{E, 13}$ is surjective or conjugate to a subgroup of $B(13)$, $N_s(13)$, $N_{ns}(13)$ or to a subgroup of $G_7$, a group whose image in $P\GL_2(\mathbb{F}_{13})$ is isomorphic to $S_4$. Let $C$ be the kernel of a $13$-isogeny. If $\rho_{E, 13}$ is surjective, we have $[\mathbb{Q}(C) : \mathbb{Q}] = 14$ by Lemma \ref{surjective}. The possibilities $N_{ns}(13)$ and $N_s(13)$ for a non-CM $E/\mathbb{Q}$ have been eliminated by \cite[Thoerem 1.1., Corollary 1.3.]{balDogra}. It was proved by \cite[Section 5.1.]{balDograExm} that the image of $\rho_{E, 13}$ is conjugate to a subgroup of $G_7$ for only three possible $j$-invariants of $E$. Also, by \cite[Theorem 1.8.]{reprZywina}, the image is exactly $G_7$ in those cases. If the image of $\rho_{E, 13}$ is $G_7$, we can use \cite[Table 2]{najGrowth} to see that if the image of $\rho_{E, 13}$ is conjugate to $G_7$, then $[\mathbb{Q}(P) : \mathbb{Q}] \geq 72$ for any $P$ of order $13$. By putting $C = \langle P \rangle$, we can use Lemma \ref{interdeg} to get: $[\mathbb{Q}(C) : \mathbb{Q}] = \frac{[\mathbb{Q}(P) : \mathbb{Q}]}{[\mathbb{Q}(P) : \mathbb{Q}(C)]} \geq 6$. Otherwise, there is a $13$-isogeny is defined over $\mathbb{Q}$, so by \cite[Table 3]{lozanoTable}, we know that $j(E) = \frac{(h^2+5h+13)(h^4+7h^3+20h^2+19h+1)^3}{h}$ for some $h \in \mathbb{Q}$.

\subsection{Case $11 \mid n$} \label{data11}
We can use \cite[Theorem 1.6]{reprZywina} which tells us that the image of $\rho_{E, 11}$ is surjective, conjugate to a subgroup of $B(11)$ or to a subgroup of $N_{ns}(11)$. Let $C$ be the kernel of an $11$-isogeny. If $\rho_{E, 11}$ is surjective, we have $[\mathbb{Q}(C) : \mathbb{Q}] = 12$ by Lemma \ref{surjective}. If the image of $\rho_{E, 11}$ is conjugate to a subgroup of $N_{ns}(11)$, we can use \cite[Table 1]{najGrowth} to see that in that case we have $[\mathbb{Q}(P) : \mathbb{Q}] = 120$ for any $P$ of order $11$. By putting $C = \langle P \rangle$, we can use Lemma \ref{interdeg} to get: $[\mathbb{Q}(C) : \mathbb{Q}] = \frac{[\mathbb{Q}(P) : \mathbb{Q}]}{[\mathbb{Q}(P) : \mathbb{Q}(C)]} \geq 12$. Otherwise, there is an $11$-isogeny is defined over $\mathbb{Q}$, so by \cite[Table 4]{lozanoTable}, we know that $j(E) \in \{-11 \cdot 131^3, -11^2\}$.

\subsection{Case $7 \mid n$} \label{data7}
We can use \cite[Theorem 1.5]{reprZywina} which tells us that the image of $\rho_{E, 7}$ is surjective, conjugate to a subgroup of $B(7)$, $N_{ns}(7)$ or $N_s(7)$. Let $C$ be the kernel of a $7$-isogeny. If the image of $\rho_{E, 7}$ is surjective or conjugate to a subgroup of $N_{ns}(7)$, we can use Lemmas \ref{surjective} and \ref{nonsplit} respectively to get that $[\mathbb{Q}(C) : \mathbb{Q}] \geq 8$. If the image of $\rho_{E, 7}$ is conjugate to a subgroup of $N_s(7)$, then $E$ has a cyclic $7$-isogeny defined over a quadratic extension of $\mathbb{Q}$ since a split-Cartan subgroup has index $2$ in its normalizer. Otherwise, there is a $7$-isogeny is defined over $\mathbb{Q}$, so by \cite[Table 3]{lozanoTable}, we know that $j(E) = \frac{(h^2+13h+49)(h^2 + 5h + 1)^3}{h}$ for some $h \in \mathbb{Q}$.

\subsection{Case $5 \mid n$} \label{data5}
We can use \cite[Theorem 1.4]{reprZywina} which tells us that the image of $\rho_{E, 5}$ is surjective, conjugate to a subgroup of $B(5)$, $N_{ns}(5)$, $N_s(5)$ or to a group $G_9$ which is a unique maximal subgroup of $\GL_2(\mathbb{F}_5)$ containing $N_s(5)$. Let $C$ be the kernel of a $5$-isogeny. If the image of $\rho_{E, 5}$ is surjective, we can use Lemma \ref{surjective} to get that $[\mathbb{Q}(C) : \mathbb{Q}] = 6$. If $\rho_{E, 5}$ is conjugate to a subgroup of $N_{ns}(5)$ or to $G_9$, we can use \cite[Table 2]{najGrowth} to see that in those cases we have $[\mathbb{Q}(P) : \mathbb{Q}] = 24$ for any $P$ of order $5$. By putting $C = \langle P \rangle$, we can use Lemma \ref{interdeg} to get: $[\mathbb{Q}(C) : \mathbb{Q}] = \frac{[\mathbb{Q}(P) : \mathbb{Q}]}{[\mathbb{Q}(P) : \mathbb{Q}(C)]} \geq 6$. If $\rho_{E, 5}$ is conjugate to a subgroup of $N_s(5)$, then $E$ has a cyclic $5$-isogeny defined over a quadratic extension of $\mathbb{Q}$ since a split-Cartan subgroup has index $2$ in its normalizer. Otherwise, $E$ has a $5$-isogeny defined over $\mathbb{Q}$. Hence, by \cite[Table 3]{lozanoTable}, we know that
$j(E) = \frac{(h^2 + 10h + 5)^3}{h}$ for some $h \in \mathbb{Q}$.

\subsection{Case $3 \mid n$} \label{data3}
We can use \cite[Theorem 1.2]{reprZywina} which tells us that the image of $\rho_{E, 3}$ is surjective or conjugate to a subgroup of $B(3)$, $N_{ns}(3)$ or $N_s(3)$. Let $C$ be the kernel of a $3$-isogeny. If the image of $\rho_{E, 3}$ is surjective, we can use Lemma \ref{surjective} to get that $[\mathbb{Q}(C) : \mathbb{Q}] = 4$. If $\rho_{E, 3}$ is conjugate to a subgroup of $N_{ns}(3)$, we can use \cite[Table 2]{najGrowth} to see that in those cases we have $[\mathbb{Q}(P) : \mathbb{Q}] = 8$ for any $P$ of order $3$. By putting $C = \langle P \rangle$, we can use Lemma \ref{interdeg} to get: $[\mathbb{Q}(C) : \mathbb{Q}] = \frac{[\mathbb{Q}(P) : \mathbb{Q}]}{[\mathbb{Q}(P) : \mathbb{Q}(C)]} \geq 4$. We can also see from \cite[Theorem 1.2]{reprZywina} that if $\rho_{E, 3}$ is not surjective, then either $E$ has a $3$-isogeny defined over $\mathbb{Q}$, or $j(E) = h^3$ for $h \in \mathbb{Q}$.

\subsection{Case $2 \mid n$} \label{data2}
We can use \cite[Theorem 1.1]{reprZywina} to see that either $\rho_{E, 2}$ is surjective, $E$ has a $2$-isogeny over $\mathbb{Q}$ or that $j(E) = h^2 + 1728$ for some $h \in \mathbb{Q}$. Notice that if $\rho_{E, 2}$ is surjective, we can again use Lemma \ref{surjective} to get that $[\mathbb{Q}(C) : \mathbb{Q}] = 3$.

\section{Two different prime divisors of isogeny degree}

In this section we will consider a situation when $E/\mathbb{Q}$ without CM has a cyclic $n$-isogeny defined over a quadratic extension of $\mathbb{Q}$ and $n$ has at least two different prime divisors $p < q$. It is known that if we only consider the isogenies defined over $\mathbb{Q}$, then $(p, q) \in \{(2, 3), (2, 5), (3, 5), (3, 7)\}$. Notice that in the statement of the below Lemma, we allow pairs $(2, 7)$ and $(7, 13)$ to potentially occur, but we will eliminate them in the later chapters.

\begin{lemma} \label{diffPrimes}
Let $E/ \mathbb{Q}$ be a non-CM elliptic curve with a cyclic $n$-isogeny defined over a quadratic number field $K$. Assume that $n$ has at least two different prime divisors $p$ and $q$ with $p < q$. Then all possible pairs $(p, q)$ are the same ones as for $\mathbb{Q}$-rational isogenies plus the pairs $(2, 7)$ and $(7, 13)$.
\end{lemma}

\begin{proof}
Clearly, $E$ has a $p$-isogeny and a $q$-isogeny over $K$. We will constantly be using the conclusions from Section \ref{basicRes}.

\textbf{Case $q = 37$:} We know from subsection \ref{data37} that $j(E) \in \{-7 \cdot 11^3, -7 \cdot 137^3 \cdot 2083^3\}$ and there is a $37$-isogeny defined over $\mathbb{Q}$. If $p = 17$, then we see from subsection \ref{data17} that $j(E) \in \{-17 \cdot 373^3 / 2^{17}, -17^2 \cdot 101^3 / 2\}$, so this is impossible. If $p = 13$, then we see from subsection \ref{data13} that we must have a $13$-isogeny defined over $\mathbb{Q}$, so we have a $481$-isogeny over $\mathbb{Q}$, but that is impossible. If $p = 11$, then we see from subsection \ref{data11} that we must have a $11$-isogeny defined over $\mathbb{Q}$, so we have a $143$-isogeny over $\mathbb{Q}$, but that is impossible.

If $p = 7$, we see from subsection \ref{data7} that either $E$ has a $7$-isogeny defined over $\mathbb{Q}$ or over a quadratic extension. From subsection \ref{data7} we see that $j(E) = \frac{(h^2 + 13h + 49)(h^2 + 5h + 1)^3}{h}$ for some $h$ in some (at most) quadratic extension of $\mathbb{Q}$. We match the above formula for $j$-invariant with the two possible $j$-invariants that allow a rational $37$-isogeny and in both cases we get a polynomial over $\mathbb{Q}$ with no quadratic roots, so $h$ can't be from a quadratic extension.

If $p = 5$, we can do the analogous computation as for $p=7$. We have $j(E) = \frac{(h^2 + 10h + 5)^3}{h}$ for some $h$ inside some (at most) quadratic extension of $\mathbb{Q}$. We match the $j$-invariants like before and again we always get a polynomial over $\mathbb{Q}$ with no quadratic roots, so $h$ can't be from a quadratic extension of $\mathbb{Q}$.

If $p = 3$, then we see from subsection \ref{data3} that either $E$ has a $3$-isogeny defined over $\mathbb{Q}$ or $j(E) = h^3$. If $E$ had a $3$-isogeny defined over $\mathbb{Q}$, it would have a $111$-isogeny over $\mathbb{Q}$, which is impossible. The remaining option is that $j(E) = h^3$ for some $h \in \mathbb{Q}$, which can't match the $j$-invariants allowing a rational $37$-isogeny.

If $p = 2$, then we see from subsection \ref{data2} that either $E$ has a $2$-isogeny defined over $\mathbb{Q}$ or $j(E) = h^2 + 1728$. If $E$ had a $2$-isogeny defined over $\mathbb{Q}$, it would have a $74$-isogeny over $\mathbb{Q}$, which is impossible. The remaining option is that $j(E) = h^2 + 1728$ for some $h \in \mathbb{Q}$, which can't match the $j$-invariants allowing a rational $37$-isogeny.

\textbf{Case $q = 17$:} We know from subsection \ref{data17} that $j(E) \in \{-17 \cdot 373^3/2^{17}, -17^2 \cdot 101^3/2\}$ and there is a $17$-isogeny defined over $\mathbb{Q}$. If $p = 13$, then we see from subsection \ref{data13} that we must have a $13$-isogeny defined over $\mathbb{Q}$, so we have a $221$-isogeny over $\mathbb{Q}$, but that is impossible. If $p = 11$, then we see from subsection \ref{data13} that we must have a $11$-isogeny defined over $\mathbb{Q}$, so we have a $187$-isogeny over $\mathbb{Q}$, but that is impossible.

If $p = 7$, we see from subsection \ref{data7} that either $E$ has a $7$-isogeny defined over $\mathbb{Q}$ or over a quadratic extension. That means that, by \cite[Table 3]{lozanoTable}, $j(E) = \frac{(h^2 + 13h + 49)(h^2 + 5h + 1)^3}{h}$ for some $h$ in some (at most) quadratic extension of $\mathbb{Q}$. We match the above formula for $j$-invariant with the two possible $j$-invariants that allow a rational $17$-isogeny and in both cases we get a polynomial over $\mathbb{Q}$ with no quadratic roots, so $h$ can't be from a quadratic extension.

If $p = 5$, we can do the analogous computation as for $p=7$. We have $j(E) = \frac{(h^2 + 10h + 5)^3}{h}$ for some $h$ inside some (at most) quadratic extension of $\mathbb{Q}$. We match the $j$-invariants like before and again we always get a polynomial over $\mathbb{Q}$ with no quadratic roots, so $h$ can't be from a quadratic extension of $\mathbb{Q}$.

If $p = 3$, then we see from subsection \ref{data3} that either $E$ has a $3$-isogeny defined over $\mathbb{Q}$ or $j(E) = h^3$. If $E$ had a $3$-isogeny defined over $\mathbb{Q}$, it would have a $111$-isogeny over $\mathbb{Q}$, which is impossible. The remaining option is that $j(E) = h^3$ for some $h \in \mathbb{Q}$, which can't match the $j$-invariants allowing a rational $17$-isogeny.

If $p = 2$, then we see from subsection \ref{data2} that either $E$ has a $2$-isogeny defined over $\mathbb{Q}$ or $j(E) = h^2 + 1728$. If $E$ had a $2$-isogeny defined over $\mathbb{Q}$, it would have a $34$-isogeny over $\mathbb{Q}$, which is impossible. The remaining option is that $j(E) = h^2 + 1728$ for some $h \in \mathbb{Q}$, which can't match the $j$-invariants allowing a rational $17$-isogeny.

\textbf{Case $q = 13$:} We know from subsection \ref{data13} that the $13$-isogeny is defined over $\mathbb{Q}$ and that
\[
j(E) = \frac{(h^2+5h+13)(h^4+7h^3+20h^2+19h+1)^3}{h}
\]
for some $h \in \mathbb{Q}$. If $p = 11$, then we know from subsection \ref{data11} that the $11$-isogeny must be defined over $\mathbb{Q}$, so $E$ has a $143$-isogeny over $\mathbb{Q}$, a contradiction. If $p = 7$, then that case is more difficult and we will only solve it near the end of this paper, see Section \ref{91iso}.

If $p = 5$, then $E$ has a $65$-isogeny defined over a quadratic extension of $\mathbb{Q}$. We can use the result from \cite[Section 4]{Box19} about quadratic points on $X_0(65)$. The result states that all quadratic points on $X_0(65)$ are coming from $X_0(65)^{+}(\mathbb{Q})$ via quotient map $\rho : X_0(65) \mapsto X_0(65)^{+}$. Notice that $X_0(65)(\mathbb{Q})$ contains no non-cuspidal points, so we can assume that $E$ is represented by some quadratic, but not rational point $Q$ on $X_0(65)$. If $Q$ represents the pair $(E, C)$, then the point $w_{65}(Q)$ is the same as $Q^{\sigma}$ (Galois conjugate) and represents a pair $(E^{\sigma}, C')$, where $E$ and $E^{\sigma}$ are $65$-isogenous. Since $E$ is defined over $\mathbb{Q}$, we have $E \cong E^{\sigma}$ and $E$ is $65$-isogenous to itself, hence it has CM, contradiction.

If $p = 3$, then we see from subsection \ref{data3} that either $E$ has a $3$-isogeny defined over $\mathbb{Q}$ or $j(E) = h^3$. If $E$ had a $3$-isogeny defined over $\mathbb{Q}$, it would have a $39$-isogeny over $\mathbb{Q}$, which is impossible. The remaining option is that $j(E) = h^3$ for some $h \in \mathbb{Q}$. We match that formula to the above formula for $j$-invariants allowing a rational $13$-isogeny. We obtain a genus $2$ curve with a Jacobian of rank $0$. By using built-in \texttt{Chabauty0()} function in Magma, we easily get that our curve has only one rational point which doesn't give us the desired elliptic curve.

If $p = 2$, then we see from subsection \ref{data2} that either $E$ has a $2$-isogeny defined over $\mathbb{Q}$ or $j(E) = h^2 + 1728$. If $E$ had a $2$-isogeny defined over $\mathbb{Q}$, it would have a $26$-isogeny over $\mathbb{Q}$, which is impossible. The remaining option is that $j(E) = h^2 + 1728$ for some $h \in \mathbb{Q}$. We match that formula to the above formula for $j$-invariants allowing a rational $13$-isogeny. This time we get a genus $1$ curve containing only one rational point which doesn't give us the desired elliptic curve.

\textbf{Case $q = 11$:} We know from subsection \ref{data11} that the $11$-isogeny is defined over $\mathbb{Q}$ and that $j(E) \in \{-11 \cdot 131^3, -11^2\}$. If $p = 7$, we see from subsection \ref{data7} that $j(E) = \frac{(h^2 + 13h + 49)(h^2 + 5h + 1)^3}{h}$ for $h$ in (at most) quadratic extension of $\mathbb{Q}$. We again get a polynomial with no quadratic roots by matching that formula to the possible $j$-invariants allowing a rational $11$-isogeny. Hence, this is impossible.

If $p = 5$, we see from subsection \ref{data5} that $j(E) = \frac{(h^2 + 10h + 5)^3}{h}$ for $h$ in (at most) quadratic extension of $\mathbb{Q}$. We again get a polynomial with no quadratic roots by matching that formula to the possible $j$-invariants allowing a rational $11$-isogeny. Hence, this is impossible.

If $p = 3$, then we see from subsection \ref{data3} that either $E$ has a $3$-isogeny defined over $\mathbb{Q}$ or $j(E) = h^3$ for $h \in \mathbb{Q}$. If $E$ had a $3$-isogeny defined over $\mathbb{Q}$, it would have a $33$-isogeny over $\mathbb{Q}$, which is impossible. The remaining option is that $j(E) = h^3$ for some $h \in \mathbb{Q}$. We match that formula to the possible $j$-invariants allowing a rational $11$-isogeny and we easily see that this case is impossible.

If $p = 2$, then we see from subsection \ref{data2} that either $E$ has a $2$-isogeny defined over $\mathbb{Q}$ or $j(E) = h^2 + 1728$. If $E$ had a $2$-isogeny defined over $\mathbb{Q}$, it would have a $22$-isogeny over $\mathbb{Q}$, which is impossible. The remaining option is that $j(E) = h^2 + 1728$ for some $h \in \mathbb{Q}$. We match that formula to the possible $j$-invariants allowing a rational $11$-isogeny and we easily see that this case is impossible.

\textbf{Case $q = 7$:} The only situation we have to eliminate here is $p = 5$. If $p = 5$, we can use a similar argument as in the case $(p, q) = (5, 13)$. We use the result from \cite[Table 9]{bnQcurve}. It is well-known that there are no non-cuspidal rational points on $X_0(35)$. Any non-exceptional quadratic point $P$ can be paired up with its Galois conjugate $P^{\sigma}$ which is equal to $w_{35}(P)$. If $P$ represented some $E$ with a rational $j$-invariant, then $w_{35}(P)$ would represent $E^{\sigma}$ which is $35$-isogenous to $E$. Since $j(E) = j(E^{\sigma})$, we have that $E$ is $35$-isogenous to itself, so it has CM, a contradiction. There is only one exceptional quadratic point on $X_0(35)$ and it corresponds to a CM curve. Therefore, this case is impossible. This completes the proof of Lemma \ref{diffPrimes}.

\end{proof}

\section{Non-squarefree isogeny degrees}
\label{sqPrime}
In this section we will consider a situation when $E/\mathbb{Q}$ without CM has an $n$-isogeny defined over a quadratic extension of $\mathbb{Q}$ and $n$ is divisible by $p^2$ for some prime $p$. It is known that if we only consider the isogenies defined over $\mathbb{Q}$, then $p \in \{2,3,5\}$.

\begin{proposition} \label{7iso}
Let $E/ \mathbb{Q}$ be a non-CM elliptic curve with a cyclic $n$-isogeny defined over a quadratic number field $K$. Assume that $p^2 \mid n$ for some prime $p$. Then $p \in \{2, 3, 5\}$.
\end{proposition}

\begin{proof}
Clearly, $E$ has a cyclic $p^2$-isogeny and a cyclic $p$-isogeny defined over $K$.

Assume $p > 7$. Then we know from section \ref{basicRes} that our $p$-isogeny has to be defined over $\mathbb{Q}$. We can use \cite[Theorem 3.9]{lombTron} to conclude that if $E$ has a $p$-isogeny over $\mathbb{Q}$, then the image of $\rho_{E, p^{\infty}}$ contains a Sylow pro-$p$ subgroup of $\GL_2(\mathbb{Z}_p)$.

Every Sylow pro-$p$ subgroup of $\GL_2(\mathbb{Z}_p)$ is conjugate to this specific Sylow pro-$p$ subgroup:
\[
S = \Bigg\{
\begin{pmatrix}
a & b \\
c & d
\end{pmatrix} \in \GL_2(\mathbb{Z}_p) \mid a \equiv d \equiv 1 \pmod{p}, \quad c \equiv 0 \pmod{p}
\Bigg\}.
\]
so we can choose compatible bases for all $E[p^k]$ such that the image of $\rho_{E, p^{\infty}}$ contains $S$.

This means that $p$-adic representation $\rho_{E, p^{\infty}}$ is defined modulo $p$ (see definition \ref{definedModP}). Now we can use Lemma \ref{definedP} to conclude that for any cyclic subgroup $C$ of $E(\overline{\mathbb{Q}})$ of order $p^2$, we have $[\mathbb{Q}(C) : \mathbb{Q}(pC)] = p$, so any $p^2$-isogeny has to be defined over a field of degree at least $p > 7$.

If $p = 7$ and $E$ has a rational $7$-isogeny, we can repeat the identical conclusions as above. Otherwise, we know from subsection \ref{data7} that the image of $\rho_{E, 7}$ is conjugate to a subgroup of $N_s(7)$.

There are three such possible images, two of which only appear when $j(E) = 2268945/128$, according to \cite[Theorem 1.5]{reprZywina}.

If we have $j(E) = 2268945/128$, we can use the classical modular polynomial $\Phi_{N}(X, Y)$. It is known from \cite{igusa} that for a field $F$ of characteristic not dividing $N$, a non-CM elliptic curve $E/F$ has a cyclic $N$-isogeny if and only if $\Phi_{N}(X, j(E))$ has a zero in $F$. We can factorize $\Phi_{49}(X, 2268945/128)$ into three irreducible factors of degrees $14$, $14$, $21$ respectively. Therefore, a cyclic $49$-isogeny is defined over a number field of degree (at least) $14$.

The third possible image is the whole $N_s(7)$. We use \texttt{Magma} to check all subgroups of $\GL_2(\mathbb{Z}/49\mathbb{Z})$ and select only those which reduce modulo $7$ to $N_s(7)$, all up to conjugation.

There are $8$ such subgroups of $\GL_2(\mathbb{Z}/49\mathbb{Z})$ up to conjugation. Call them $H_i$ for $i \in \{1, 2, \ldots, 8\}$. By using \cite[Theroem 3.16]{lombTron}, we see that all $H_i$ must contain all scalars congruent to $1$ modulo $7$. This property is clearly not affected by conjugation.

Group $H_1$ is of order $72$ and contains one scalar congruent to $1$ modulo $7$, so it is eliminated.

Group $H_2$ is of order $504$ and contains one scalar congruent to $1$ modulo $7$, so it is eliminated.

Group $H_3$ is of order $504$ and is conjugate to a subgroup of $N_s(49)$.

Group $H_4$ is of order $3528$ and contains one scalar congruent to $1$ modulo $7$, so it is eliminated.

Group $H_5$ is of order $3528$ and is conjugate to $N_s(49)$.

Group $H_6$ is of order $24696$ and contains one scalar congruent to $1$ modulo $7$, so it is eliminated.

Groups $H_7$ and $H_8$ are of orders $24696$ and $172872$ respectively and they act on the cyclic subgroups of $E[49]$ of order $49$. The corresponding orbit lengths are $14$ and $42$ in both cases, so a cyclic $49$-isogeny is defined over the field of degree (at least) $14$. The only subgroups we still didn't eliminate are those conjugate to some subgroup of $N_s(49)$. If there exists a non-CM elliptic curve over $\mathbb{Q}$ such that its mod $49$ representation falls into that category, it will be represented by a point on $X_{s}(49)(\mathbb{Q})$. Recall that $X_s(n)$ is the modular curve associated to the normalizer of the split Cartan subgroup.

We can, for example, use the comment from \cite{parentCom} which relies on \cite{moduliKM} to recall that there is a $\mathbb{Q}$-isomorphism $X_{s}(N) \cong X_0(N^2) / w_{N^2} \equiv X_0^{+}(N^2)$, where $w_{N^2}$ is the Atkin-Lehner involution. One can also look at \cite[Section 2]{biluParReb} for the modular interpretation of the aforementioned isomorphism to see that CM points and cusps on $X_{s}(p^r)$ correspond to CM points and cusps on $X_0^{+}(p^{2r})$ for a prime $p$. We know from \cite[Theorem 3.14]{momoShim} that $X_0^{+}(7^r)(\mathbb{Q})$ consists only of cusps and CM-points for $r \geq 3$. Since we were considering $X_{s}(7^2) \cong X_0(7^4) / w_{7^4} \equiv X_0^{+}(7^4)$, we are done. Cases $p \in \{2, 3, 5\}$ are already possible over $\mathbb{Q}$. Therefore, this completes the proof. 

\end{proof}

\section{Isogenies of degree $5^k$}\label{5isogenies}

We will now consider the isogenies divisible by powers of $5$. We will use the following theorem:

\begin{tm}[{\cite[Theorem 2]{greenIsog}}] \label{isog5}
Let $E / \mathbb{Q}$ be a non-CM elliptic curve with an isogeny of degree $5$ defined over $\mathbb{Q}$. If none of the elliptic curves in the $\mathbb{Q}$-isogeny class of $E$ has two independent isogenies of degree $5$, then the image of $\rho_{E, 5^{\infty}}$ contains a Sylow pro-$5$ subgroup of $\GL_2(\mathbb{Z}_5)$. Otherwise, the index $[\GL_2(\mathbb{Z}_5) : Im(\rho_{E, 5^{\infty}})]$ is divisible by $5$, but not by $25$. 
\end{tm}

We prove the following Lemma which considers the situation when $E$ has a rational $5$-isogeny:

\begin{lemma} \label{yes5iso}
Let $E/ \mathbb{Q}$ be a non-CM elliptic curve with a cyclic $n$-isogeny defined over a number field $K$ such that $5^k \mid n$ with $k \geq 2$. Assume $E$ also has a $5$-isogeny defined over $\mathbb{Q}$. Then $[K : \mathbb{Q}] \geq 5^{k-2}$.
\end{lemma}

\begin{proof}

Clearly, $E$ has a cyclic $5^k$-isogeny defined over $K$. If the $\mathbb{Q}$-isogeny class of $E$ does not contain a curve with two independent $5$-isogenies, we can use theorem \ref{isog5} to conclude that the image of $\rho_{E, 5^{\infty}}$ contains a Sylow pro-$5$ subgroup of $\GL_2(\mathbb{Z}_5)$. We can conclude that $\rho_{E, 5^{\infty}}$ is defined modulo $5$ (similar as in Proposition \ref{7iso}), so we can use Lemma \ref{definedP} to get that if $C$ is a cyclic subgroup of $E(\overline{\mathbb{Q}})$ of order $5^k$ ($k \geq 2$), then $[\mathbb{Q}(C) : \mathbb{Q}(5C)] = 5$. Therefore, if the $\mathbb{Q}$-isogeny class of $E$ does not contain a curve with two independent $5$-isogenis, any cyclic $5^k$-isogeny is defined over the number field of degree at least $5^{k-1}$ so $[K : \mathbb{Q}] \geq 5^{k-1}$.

Now assume that the $\mathbb{Q}$-isogeny class of $E$ contains a curve $E'$ with two independent $5$-isogenies. Recall the fact that there is a cyclic isogeny $\phi: E \mapsto E'$ of a unique degree $d$, since they are non-CM curves. We consider two cases, depending on whether $5 \mid d$.

\textbf{Case $5 \nmid d$:} First we show that the images of $\rho_{E, 5}$ and $\rho_{E', 5}$ are the same, up to conjugation. Let $\{P, Q\}$ be a basis for $E[5]$. Then $\{\phi(P), \phi(Q)\}$ is a basis for $E'[5]$ since $5 \nmid d$. Notice that if $P^{\sigma} = aP + bQ$ for $\sigma \in G_{\mathbb{Q}}$, then, since $\phi$ is defined over $\mathbb{Q}$, we have $\phi(P)^{\sigma} = \phi(P^{\sigma}) = a\phi(P) + b\phi(Q)$. We get the analogous result for $Q$ so we get that $\rho_{E, 5}(\sigma) = \rho_{E', 5}(\sigma)$. Therefore, the images of $\rho_{E, 5}$ and $\rho_{E', 5}$ are the same, up to conjugation. This means that $E$ also contains two independent $5$-isogenies, so $Im(\rho_{E, 5})$ consists of diagonal matrices (up to conjugation). The subgroup of diagonal matrices in $\GL_2(\mathbb{Z} /5\mathbb{Z})$ has $16$ elements, so $\#Im(\rho_{E, 5}) \mid 16$. Let $\pi_5 : \GL_2(\mathbb{Z}_5) \mapsto \GL_2(\mathbb{Z}/5\mathbb{Z})$ be the mod $5$ reduction. We have from the first isomorphism theorem:
\begin{align*}
    &[Im(\rho_{E, 5^{\infty}}) : Im(\rho_{E, 5^{\infty}}) \cap ker(\pi_5)] = \# Im(\rho_{E, 5}), \\
    &[\GL_2(\mathbb{Z}_5) : ker(\pi_5)] = \# \GL_2(\mathbb{Z}/5\mathbb{Z}) = 480.
\end{align*}
We know that $Im(\rho_{E, 5^{\infty}}) \cap ker(\pi_5)$ is of finite index in $Im(\rho_{E, 5^{\infty}})$ and that $Im(\rho_{E, 5^{\infty}})$ is of finite index in $\GL_2(\mathbb{Z}_5)$ due to Serre's open image theorem. Therefore, $Im(\rho_{E, 5^{\infty}}) \cap ker(\pi_5)$ is of finite index in $\GL_2(\mathbb{Z}_5)$ and consequently in $ker(\pi_5)$ too. The group $ker(\pi_5)$ is a pro-$5$ group, so any of its subgroups of finite index has index which is a power of $5$ (see \cite[Theorem 1]{proP}). Now we have for some $m \geq 0$:
\begin{align*}
    &[\GL_2(\mathbb{Z}_5) : Im(\rho_{E, 5^{\infty}}) \cap ker(\pi_5)] = [\GL_2(\mathbb{Z}_5) : ker(\pi_5)][ker(\pi_5) : Im(\rho_{E, 5^{\infty}}) \cap ker(\pi_5)] = 480 \cdot 5^m
\end{align*}
We also have:
\begin{align*}
&480 \cdot 5^m = [\GL_2(\mathbb{Z}_5) : Im(\rho_{E, 5^{\infty}}) \cap ker(\pi_5)] =\\
&= [\GL_2(\mathbb{Z}_5) : Im(\rho_{E, 5^{\infty}})][Im(\rho_{E, 5^{\infty}}) : Im(\rho_{E, 5^{\infty}}) \cap ker(\pi_5)] =\\
&= \#Im(\rho_{E, 5}) \cdot [\GL_2(\mathbb{Z}_5) : Im(\rho_{E, 5^{\infty}})].
\end{align*}
We know that $\#Im(\rho_{E, 5}) \mid 16$ and from theorem \ref{isog5} we know that $25 \nmid [\GL_2(\mathbb{Z}_5) : Im(\rho_{E, 5^{\infty}})]$. Therefore, $m = 0$, so $ker(\pi_5) \leq Im(\rho_{E, 5^{\infty}})$. This means that $\rho_{E, 5^{\infty}}$ is defined modulo $5$. We can now use Lemma \ref{definedP} like before and we get that cyclic $5^k$-isogeny is defined over the number field of degree at least $5^{k-1}$ so $[K : \mathbb{Q}] \geq 5^{k-1}$.

\textbf{Case $5 \mid d$:} Our first step in this case is to compose $\phi$ with one of the rational $5$-isogenies on the curve $E'$ such that the composition is still a cyclic isogeny. Assume $\{P, Q\}$ is the basis for $E[5]$ such that $\phi(P) = O_{E'}$. Then $\phi(Q) \neq O_{E'}$ since $\phi$ is cyclic. Also, $\phi(Q)$ is of order $5$. At least one of the two independent rational $5$-isogenies of $E'$ will not have $\phi(Q)$ in its kernel, call it $\alpha$. Then the composition $\alpha \circ \phi$ will not have $Q$ in its kernel, so $\alpha \circ \phi$ is cyclic and has degree $5d$. Also notice that $\alpha \circ \phi$ is defined over $\mathbb{Q}$ and $25 \mid 5d$ so we must have $5d = 25$. This means that $E$ has a cyclic $25$-isogeny defined over $\mathbb{Q}$. Therefore, $Im(\rho_{E, 25})$ consists of upper-triangular matrices (up to conjugation). The subgroup of upper-triangluar matrices in $\GL_2(\mathbb{Z}/25\mathbb{Z})$ has $10000$ elements, so $\#Im(\rho_{E, 25}) \mid 2^4\cdot 5^4$.

Now we proceed very similar to the case $5 \nmid d$. We set $\pi_{25} : \GL_2(\mathbb{Z}_5) \mapsto \GL_2(\mathbb{Z}/25\mathbb{Z})$ to be the mod $25$ reduction map. We have from the first isomorphism theorem:
\begin{align*}
    &[Im(\rho_{E, 5^{\infty}}) : Im(\rho_{E, 5^{\infty}}) \cap ker(\pi_{25})] = \# Im(\rho_{E, 25}), \\
    &[\GL_2(\mathbb{Z}_5) : ker(\pi_{25})] = \# \GL_2(\mathbb{Z}/25\mathbb{Z}) = 5^4 \cdot 480.
\end{align*}
We know that $Im(\rho_{E, 5^{\infty}}) \cap ker(\pi_{25})$ is of finite index in $Im(\rho_{E, 5^{\infty}})$ and that $Im(\rho_{E, 5^{\infty}})$ is of finite index in $\GL_2(\mathbb{Z}_5)$ due to Serre's open image theorem. Therefore, $Im(\rho_{E, 5^{\infty}}) \cap ker(\pi_{25})$ is of finite index in $\GL_2(\mathbb{Z}_5)$ and consequently in $ker(\pi_{25})$ too. The group $ker(\pi_{25})$ is a subgroup of $ker(\pi_{5})$ of finite index ($5^4$), hence $Im(\rho_{E, 5^{\infty}}) \cap ker(\pi_{25})$ is a subgroup of $ker(\pi_{5})$ of finite index. Since $ker(\pi_{5})$ is a pro-$5$ group, any its subgroup of finite index has index which is a power of $5$ (see \cite[Theorem 1]{proP}). Hence the index $[ker(\pi_{25}) : Im(\rho_{E, 5^{\infty}}) \cap ker(\pi_{25})]$ is also a power of $5$. Now we have for some $m \geq 0$:
\begin{align*}
    &[\GL_2(\mathbb{Z}_5) : Im(\rho_{E, 5^{\infty}}) \cap ker(\pi_{25})] =\\
    &=[\GL_2(\mathbb{Z}_5) : ker(\pi_{25})][ker(\pi_{25}) : Im(\rho_{E, 5^{\infty}}) \cap ker(\pi_{25})] = 5^4 \cdot 480 \cdot 5^m
\end{align*}
We also have:
\begin{align*}
    &5^4 \cdot 480 \cdot 5^m = [\GL_2(\mathbb{Z}_5) : Im(\rho_{E, 5^{\infty}}) \cap ker(\pi_{25})] =\\
    &= [\GL_2(\mathbb{Z}_5) : Im(\rho_{E, 5^{\infty}})][Im(\rho_{E, 5^{\infty}}) : Im(\rho_{E, 5^{\infty}}) \cap ker(\pi_{25})] =\\
    &= \#Im(\rho_{E, 25}) \cdot [\GL_2(\mathbb{Z}_5) : Im(\rho_{E, 5^{\infty}})].
\end{align*}
We know that $\#Im(\rho_{E, 25}) \mid 2^4 \cdot 5^4$ and from theorem \ref{isog5} we know that $25 \nmid [\GL_2(\mathbb{Z}_5) : Im(\rho_{E, 5^{\infty}})]$. Therefore, $m = 0$, so $ker(\pi) \leq Im(\rho_{E, 5^{\infty}})$. This means that $\rho_{E, 5^{\infty}}$ is defined modulo $25$. We can now use Lemma \ref{definedP} like before and we get that cyclic $5^k$-isogeny is defined over the number field of degree at least $5^{k-2}$ so $[K : \mathbb{Q}] \geq 5^{k-2}$. This completes the proof.
\end{proof}

Next we consider the situation when $E$ doesn't have a rational $5$-isogeny.

\begin{lemma} \label{no5iso}
Let $E/ \mathbb{Q}$ be a non-CM elliptic curve which doesn't have a $5$-isogeny defined over $\mathbb{Q}$. Then a cyclic $25$-isogeny of $E$ is defined over the number field of degree at least $6$.
\end{lemma}

\begin{proof}
We know from subsection \ref{data5} that if the image of $\rho_{E,5}$ is surjective, conjugate to $N_{ns}(5)$ or to $G_9$ from \cite[Theorem 1.4]{reprZywina}, then any $5$-isogeny (and hence cyclic $25$-isogeny) is defined over the number field of degree at least $6$. The only remaining possible images of $\rho_{E,5}$ such that $E$ doesn't have a $5$-isogeny defined over $\mathbb{Q}$ are $N_s(5)$ and one of its subgroups ($G_3$ from \cite[Theorem 1.4]{reprZywina}). We eliminate these using \texttt{Magma} the same way we did in Proposition \ref{7iso}: we check all subgroups of $\GL_2(\mathbb{Z}/25\mathbb{Z})$ and select only those which reduce modulo $5$ to $N_s(5)$ or $G_3$, all up to conjugation. Those are the possible images of $\rho_{E, 25}$. Similarly to the proof of Proposition \ref{7iso}, for each possible subgroup $H$, one of the following happens:

\begin{enumerate}
    \item Group $H$ does not contain all scalars congruent to $1$ modulo $5$, a contradiction with \cite[Theroem 3.16]{lombTron}.
    \item Orbit lengths of cyclic subgroups of $E[25]$ of order $25$ under the action of $H$ are $10$ and $20$, so a cyclic $25$-isogeny is defined over the number field of degree at least $10$.
    \item Group $H$ is conjugate to a subgroup of $N_s(25)$.
\end{enumerate}

 We can conclude that the last case is impossible by again using $X_{s}(5^2) \cong X_0^+(5^4)$ (via a $\mathbb{Q}$-isomorphism) and \cite[Theroem 3.14]{momoShim} like in the end of the proof of Proposition \ref{7iso}. This tells us that all rational points on $X_0^+(5^4)$, and hence also on $X_{s}(5^2)$, are cusps or CM points. This completes the proof. 
\end{proof}

\begin{proposition} \label{5isoFinal}
Let $E/ \mathbb{Q}$ be a non-CM elliptic curve with a cyclic $n$-isogeny defined over a quadratic number field $K$. Assume that $5^k \mid n$. Then $k \leq 2$.
\end{proposition}
\begin{proof}
This follows directly from Lemmas \ref{yes5iso} and \ref{no5iso}.
\end{proof}

\section{Isogenies of degree $3^k$}

We will now consider isogenies divisible by powers of $3$. To begin with, we will prove a simple group-theoretic Lemma to make later proofs easier:

\begin{lemma} \label{grpLemma}
Let $G$ be a group and $H$, $L$ its subgroups such that $[G : L] \leq 2$. Then $[H : H \cap L] \leq 2$.
\end{lemma}

\begin{proof}
If $[G : L] = 1$ then $G = L$ and $[H : H \cap L] = [H : H] = 1$. The same holds if $H \leq L$. Now assume $[G : L] = 2$ and $H \nsubseteq L$. Then $L \trianglelefteq G$. Hence we can use the second isomorphism theorem saying that $(HL)/L \cong H/(H \cap L)$. Since $H \nsubseteq L$ and there are only two $L$-cosets in $G$, we know that $HL = G$. Therefore, $H/(H \cap L) \cong G/L$, so $[H : H \cap L] \leq 2$.
\end{proof}

First we consider the situation when $E$ has a rational $3$-isogeny.

\begin{lemma} \label{yes3iso}
Let $E/ \mathbb{Q}$ be a non-CM elliptic curve with a cyclic $n$-isogeny defined over a number field $K$ such that $3^k \mid n$ with $k \geq 2$. Assume $E$ also has a $3$-isogeny defined over $\mathbb{Q}$. Then $[K : \mathbb{Q}] \geq 3^{k-2}$.
\end{lemma}

\begin{proof}
Clearly, $E$ has a cyclic $3^k$-isogeny defined over $K$. If $E$ has a $3$-isogeny defined over $\mathbb{Q}$, we can use the \cite[Corollary 1.3.1.]{troglavaNeman}. Notice that it tells us that $\langle Im(\rho_{E, 3^{\infty}}), -I \rangle$ is of level at most $9$, except for one case which we will solve later. For now assume $\langle Im(\rho_{E, 3^{\infty}}), -I \rangle$ is of level at most $9$. Let $\pi_9 : \GL_2(\mathbb{Z}_3) \mapsto \GL_2(\mathbb{Z}/9\mathbb{Z})$ and $\pi_3 : \GL_2(\mathbb{Z}_3) \mapsto \GL_2(\mathbb{Z}/3\mathbb{Z})$ be the mod $9$ and mod $3$ reductions. We have that $ker(\pi_9) \leq \langle Im(\rho_{E, 3^{\infty}}), -I \rangle$. We can use Lemma \ref{grpLemma} with $G := \langle Im(\rho_{E, 3^{\infty}}), -I \rangle$, $H := ker(\pi_9)$, $L := Im(\rho_{E, 3^{\infty}})$ to get that $[ker(\pi_9) : ker(\pi_9) \cap Im(\rho_{E, 3^{\infty}})] \leq 2$. The group $ker(\pi_9)$ is a subgroup of $ker(\pi_3)$ of finite index ($3^4$), hence $ker(\pi_9) \cap Im(\rho_{E, 3^{\infty}})$ is a subgroup of $ker(\pi_{3})$ of finite index. Since $ker(\pi_{3})$ is a pro-$3$ group, any subgroup of finite index has index which is a power of $3$ (see \cite[Theorem 1]{proP}). Hence, $[ker(\pi_9) : ker(\pi_9) \cap Im(\rho_{E, 3^{\infty}})] = 1$. Therefore, $ker(\pi_9) \leq Im(\rho_{E, 3^{\infty}})$, so $\rho_{E, 3^{\infty}}$ is defined modulo $9$. Now we can use Lemma \ref{definedP} to see that if $E$ has a cyclic $3^k$-isogeny defined over $K$ for $k \geq 2$, then $[K : \mathbb{Q}] \geq 3^{k-2}$.

Recall that there is still one group of level $27$ that $\langle Im(\rho_{E, 3^{\infty}}), -I \rangle$ can be conjugate to. We can find its generators in \cite[Table 1]{lAdicLaki}. Either $Im(\rho_{E, 3^{\infty}})$ contains $-I$ and is therefore equal to the mentioned group, or it is a subgroup of index $2$ which doesn't contain $-I$. Using \texttt{Magma}, we see that orbit lengths of cyclic subgroups of $E[27]$ of order $27$ are $3, 6, 27$ in all cases, so cyclic $27$-isogeny is defined over the field of degree at least $3$. Since $\rho_{E, 3^{\infty}}$ is defined modulo $27$ in this case, we can again use Lemma \ref{definedP} and conclude that if $K$ is the field of definition of some cyclic $3^k$-isogeny with $k \geq 2$, we have $[K : \mathbb{Q}] \geq 3^{k-2}$. This completes the proof.
\end{proof}

Now we consider the situation when $E$ doesn't have a rational $3$-isogeny.

\begin{lemma} \label{no3iso}
Let $E/ \mathbb{Q}$ be a non-CM elliptic curve which doesn't have a $3$-isogeny defined over $\mathbb{Q}$. Then a cyclic $9$-isogeny of $E$ is defined over the number field of degree at least $4$.
\end{lemma}

\begin{proof}
We know from subsection \ref{data3} that if $\rho_{E, 3}$ is surjective or if its image is conjugate to $N_{ns}(3)$, then any $3$-isogeny is defined over the number field of degree $4$. Hence, any cyclic $9$-isogeny is defined over the number field of degree at least $4$. The remaining option is that the image of $\rho_{E, 3}$ is conjugate to $N_s(3)$. In that case, $3$-isogeny is defined over a number field of degree $2$. We analyze this the same way we did for the similar situation with $5$-isogeny and $7$-isogeny. We consider all possible images of $\rho_{E, 9}$. Those are the subgroups of $\GL_2(\mathbb{Z}/9\mathbb{Z})$ that reduce to $N_s(3)$ modulo $3$, up to conjugation. Using \texttt{Magma}, we see that there are $12$ such subgroups: $8$ of them have orbit lengths of cyclic subgroups of $E[9]$ of order $9$ all equal to $6$, so any cyclic $9$-isogeny is defined over the number field of degree $6$ in those cases. The other $4$ are conjugate to a subgroup of $N_s(9)$. We can conclude that these $4$ subgroups can't appear by using \cite[Theorem 3.14]{momoShim} like in the end of the proof of Proposition \ref{7iso} and Lemma \ref{no5iso}. 
\end{proof}

\begin{proposition} \label{3isoFinal}
Let $E/ \mathbb{Q}$ be a non-CM elliptic curve with a cyclic $n$-isogeny defined over a quadratic number field $K$. Assume that $3^k \mid n$. Then $k \leq 2$.
\end{proposition}
\begin{proof}
This follows directly from Lemmas \ref{yes3iso} and \ref{no3iso}.
\end{proof}

\section{Isogenies of degree $2^k$}

We will now consider isogenies divisible by powers of $2$.

\begin{lemma} \label{2pow}
Let $E/ \mathbb{Q}$ be a non-CM elliptic curve with a cyclic $n$-isogeny defined over a number field $K$ such that $2^k \mid n$ with $k \geq 4$. Then $[K : \mathbb{Q}] \geq 2^{k-4}$. If $K$ is quadratic, then $k \leq 5$.
\end{lemma}

\begin{proof}
Clearly $E$ has a cyclic $2^k$-isogeny defined over $K$. We know from \cite[Corollary 1.3]{2AdicZB} that $\rho_{E, 2^{\infty}}$ is defined modulo $32$. We also know that a cyclic $32$-isogeny can't be defined over $\mathbb{Q}$, so it is defined over at least a quadratic extension of $\mathbb{Q}$. We can now use Lemma \ref{definedP} to conclude that $[K : \mathbb{Q}] \geq 2^{k-4}$. It is now easy to see that if $K$ is a quadratic number field, then $k \leq 5$.
\end{proof}

\section{Isogenies of degree $2^a \cdot 3^b$}

We will now consider isogenies whose degree is of the form $2^a \cdot 3^b$. Clearly, we need to only consider $a, b \geq 1$ since other cases are considered in the above sections.

\begin{lemma} \label{2and3}
Let $E/ \mathbb{Q}$ be a non-CM elliptic curve with a cyclic $n$-isogeny defined over a quadratic number field $K$, where $n = 2^a3^b$. Then $n \in \{2,4,6,8,9,12,16,18,24,32,36\}$.
\end{lemma}

\begin{proof}
We can use the Proposition \ref{3isoFinal} to conclude that $b \leq 2$.

\textbf{Case $b = 2$:} Clearly, it is enough to show that $n = 72$ is impossible. This follows directly from \cite[Table 8.13.]{OzmanSiksek19}.

\textbf{Case $b = 1$:} Clearly, it is enough to show that $n = 48$ is impossible. One can use \cite[Table 15.]{bnQcurve} to see that the only exceptional quadratic points on $X_0(48)$ are CM points. All non-exceptional quadratic points are paired up via hyperelliptic involution induced by $\beta_{48} =
\begin{pmatrix}
-6&&1 \\
-48&&6 \\
\end{pmatrix}$. We can refer to \cite[Subsection 3.4]{bnQcurve} to conclude that curves linked with this hyperelliptic involution are $12$-isogenous. Therefore, if there was an elliptic curve defined over $\mathbb{Q}$ among the obvious points, it would be $12$-isogenous to its Galois conjugate (itself) since hyperelliptic involution and Galois conjugation act identically on the non-exceptional points. Therefore, it would have CM, which is a contradiction. The proof is now complete.
\end{proof}

\section{Isogenies of degree $2^a \cdot 5^b$ over quadratic fields}

We will now consider isogenies whose degree is of the form $2^a \cdot 5^b$. Clearly, we need to only consider $a, b \geq 1$ since other cases are considered in the above sections.

\begin{lemma} \label{2and5}
Let $E/ \mathbb{Q}$ be a non-CM elliptic curve with a cyclic $n$-isogeny defined over a quadratic number field $K$, where $n = 2^a5^b$. Then $n \in \{2,4,5,8,10,16,20,25,32\}$.
\end{lemma}

\begin{proof}
We can use the Proposition \ref{5isoFinal} to conclude that $b \leq 2$.

\textbf{Case $b = 2$:} Clearly, it is enough to show that $n = 50$ is impossible. One can use \cite[Table 16.]{bnQcurve} to see that the only exceptional quadratic points on $X_0(50)$ are two CM points and four non-CM points which don't correspond to rational $j$-invariant. Non-exceptional points come in pairs via hyperelliptic involution which is also the Atkin-Lehner involution $w_{50}$. We can again deduce that those points can't correspond to rational non-CM elliptic curves like several times before.

\textbf{Case $b = 1$:} Clearly, it is enough to show that $n = 40$ is impossible. We can use \cite[Table 11]{bnQcurve}. It tells us that all exceptional quadratic points on $X_0(40)$ correspond to CM-curves. The remaining quadratic points are non-exceptional points with a rational $x$ coordinate, so the hyperelliptic involution acts the same way on them as Galois conjugation. The hyperelliptic involution $\iota$ is induced by $\beta_{40} =
\begin{pmatrix}
-10&&1 \\
-120&&10 \\
\end{pmatrix}$. We can refer to \cite[Subsection 3.4]{bnQcurve} to conclude that curves linked with this hyperelliptic involution are $20$-isogenous. Therefore, if there was an elliptic curve defined over $\mathbb{Q}$ among the non-exceptional points, it would be $20$-isogenous to itself. Therefore, it would have CM, which is a contradiction.

\end{proof}

\section{Isogenies of degree $3^a \cdot 5^b$ over quadratic fields}

We will now consider isogenies whose degree is of the form $3^a \cdot 5^b$. Clearly, we need to only consider $a, b \geq 1$ since other cases are considered in the above sections.

\begin{lemma} \label{3and5}
Let $E/ \mathbb{Q}$ be a non-CM elliptic curve with a cyclic $n$-isogeny defined over a quadratic number field $K$, where $n = 3^a5^b$. Then $n \in \{3, 5, 9, 15, 25\}$.
\end{lemma}

\begin{proof}
Clearly, it is enough to show that $n = 45$ and $n = 75$ are both impossible. This follows from \cite[Tables 8.5., 8.14.]{OzmanSiksek19}. The curve $X_0(45)$ has two quadratic CM points and four quadratic non-CM points which don't give us a rational $j$-invariant. The curve $X_0(75)$ has no non-cuspidal, non-CM quadratic points.

\end{proof}

\section{Isogenies of degree $14, 30, 63$ over quadratic fields}\label{143063}

\begin{lemma} \label{30and63}
Let $E/ \mathbb{Q}$ be a non-CM elliptic curve with a cyclic $n$-isogeny defined over a quadratic number field $K$. Then $n \not\in \{30, 63\}$.
\end{lemma}

\begin{proof}
To eliminate the option $n = 30$, we can use \cite[Table 6.]{bnQcurve} to see that $X_0(30)$ has six exceptional quadratic points. Two of them are CM points and four are non-CM points which don't correspond to a rational $j$-invariant. The remaining quadratic points are non-exceptional points which come in pairs via hyperelliptic involution $w_{15}$ and we can use the same argument as before to show that they can't correspond to a rational non-CM $j$-invariant.

To eliminate the option $n = 63$, we can use \cite[Table 8.11.]{OzmanSiksek19} to see that $X_0(63)$ has no non-CM non-cuspidal quadratic points.

\end{proof}

\begin{lemma} \label{deg14}
Let $E/ \mathbb{Q}$ be a non-CM elliptic curve with a cyclic $n$-isogeny defined over a quadratic number field $K$. Then $n \neq 14$.
\end{lemma}

\begin{proof}
Notice that $E$ can have a rational $14$-isogeny, but then $E$ has to be CM. If $E$ didn't have a rational $2$-isogeny, then any $2$-isogeny would be defined over a number field of degree $3$, making it impossible for $E$ to have a $14$-isogeny defined over quadratic number field. Hence, $E$ has a rational $2$-isogeny. This means that $E$ must have a $7$-isogeny defined over a quadratic number field, but not over $\mathbb{Q}$. By recalling subsection \ref{data7}, we see that the image of $\rho_{E, 7}$ has to be a subgroup of $N_s(7)$. We can get the form for $j$-invariant of such curves from \cite[Theorem 1.5.]{reprZywina}. We match that form with the form of the $j$-invariants allowing a rational $2$-isogeny:
\[
\frac{t(t+1)^3(t^2-5t+1)^3(t^2-5t+8)^3(t^4-5t^3+8t^2-7t+7)^3}{(t^3-4t^2+3t+1)^7} = \frac{(s+16)^3}{s}.
\]
We get a genus $3$ projective curve on which we want to find all the rational points. We map it to a curve which has a quotient that is an elliptic curve with only $6$ rational points. By taking the preimages, we find all the rational points on the starting curve, none of which give us a desired non-CM curve $E$. Those points are: $(2 : -256 : 1), (-1 : -16 : 1), (0 : -16 : 1), (0 : 1 : 0), (1 : 0 : 0)$. The last two don't give us a $j$-invariant and other give us $j$-invariants $0$ or $54000$. That completes the proof.

\end{proof}

\section{Isogenies of degree $91$ over quadratic fields}
\label{91iso}
\subsection{Method description}

Our goal is to show that there are no elliptic curves $E/\mathbb{Q}$ with a cyclic $91$-isogeny defined over a quadratic extension of $\mathbb{Q}$. We will determine all quadratic points on $X_0(91)$ up to those points that appear as pullbacks of rational points on $X_0(91)^+$ (non-exceptional points). We will see that all the exceptional points are either cusps or CM points. On the other hand, we can use the identical modular interpretation argument as several times before to show that if a non-exceptional point on $X_0(91)$ represents an $E$ with a rational $j$-invariant, then $E$ is $91$-isogenous to itself so it has CM.

We will use the relative symmetric Chabauty developed by Siksek in \cite{siksekCh} and used by Box in \cite{Box19}. We will follow the approach and the implementation of Box and easily adapt it to $X_0(91)$. Notice that our problem for $X_0(91)$ is the same problem Box tackled.

For some smooth, projective non-hyperelliptic curve $X/\mathbb{Q}$, the method provides a criterion \cite[Theorem 2.1.]{Box19} for a point on $X^{(2)}(\mathbb{Q})$ to be the only point in its residue class modulo prime $p > 2$. Also, the method provides a criterion \cite[Theorem 2.4.]{Box19} for a point on $X^{(2)}(\mathbb{Q})$ to be the only point in its residue class modulo prime $p > 2$, up to points appearing as pullbacks of points on $C(\mathbb{Q})$, where $C$ is a degree $2$ quotient of $X$.

In our case, we have $X = X_0(91)$ and $C = X_0(91)^+$. We also need $\rk(J(X)) = \rk(J(C))$, which is true in our case as both ranks are $2$. With this, we can easily find a subgroup $G \leq J_0(91)(\mathbb{Q})$ such that $2 \cdot J_0(91)(\mathbb{Q}) \leq G$, see \cite[Proprosition 3.1.]{Box19}. Also, the equality of ranks gives us an easy way of finding annihilating differentials, see \cite[Lemma 3.4.]{Box19}. For a more detailed description of this method, see \cite[Section 6]{NajmanVuk}. For the full description, see \cite{Box19} or \cite{Box21}.

\subsection{Computing the model, rank and torsion}
To get the model of $X_0(91)$, we use the approach of Özman and Siksek \cite[Section 3]{OzmanSiksek19}, but we use a different basis for the space $S_2(91)$ of weight $2$ cuspforms of level $91$. We choose a basis for $S_2(91)$ such that the matrix of $w_{91}$ is diagonal with all the diagonal elements equal to $1$ or $-1$ in that basis. This reduces the time needed to obtain the model and reduces the time needed to compute $C = X_0(91)^+$. We remove the part in the Özman-Siksek code which computes a Gröbner basis since it is only used to potentially simplify the equations and didn't seem to give us noticeable gains. With these minor adjustments, we were able to obtain models for $X_0(n)$ and the quotients $X_0(n)/w_d$ relatively quickly. To get the rank of $J_0(91)(\mathbb{Q})$, we can use the modular symbols package in \texttt{Magma} developed by W. Stein in \cite{stein00, stein07} and also the Kolyvagin-Logachev theorem \cite{KolyvaginLogachev89} identically as in \cite[Proposition 5.1.]{NajmanVuk} to get that $\rk(J_0(91)) = 2$.

To determine $J_0(91)(\mathbb{Q})_{tors}$, we will use the fact that $J_0(91)(\Q)_{tors}$ injects into $J_0(91)(\F_p)$ for an odd prime $p$ of good reduction \cite[Appendix]{katz81}. By doing this for primes $3$, $5$ and $19$, we get that $\#J_0(91)(\mathbb{Q})_{tors} \leq 336$. By taking the differences of cusps of $X_0(91)$, we are able to generate a torsion subgroup isomorphic to $\mathbb{Z}/2\mathbb{Z} \times \mathbb{Z}/168\mathbb{Z}$. Hence, $J_0(91)(\mathbb{Q})_{tors} \cong \mathbb{Z}/2\mathbb{Z} \times \mathbb{Z}/168\mathbb{Z}$.

\subsection{Computations on quotient}
We can easily compute the degree $2$ quotient $X_0(91)^+$ in \texttt{Magma} which is a genus $2$ hyperelliptic curve. We can use Stoll's algorithm \cite{stoll02} to determine the generators of the free part of $J_0(91)^+(\mathbb{Q})$ which also has rank $2$. By taking their pullbacks, we are able to generate a subgroup $G \leq J_0(91)(\mathbb(Q))$ such that $2 \cdot J_0(91)(\mathbb{Q}) \leq G$.

\subsection{Finding some quadratic points}
$X_0(91)$ has $4$ cusps which are defined over $\mathbb{Q}$. We can get $8$ more Galois-conjugate pairs of quadratic points on $X_0(91)$ by taking pullbacks of the $10$ known rational points on $X_0(91)^+$. Although we conjecture that there are no other rational points on $X_0(91)^+$, we will not prove that here. We can get one more pair of Galois-conjugate quadratic points by examining the fixed points of $w_{91}$. Two of them give us that pair. Notice that those points are CM points as they represent elliptic curves $91$-isogenous to themselves.

\subsection{Chabauty computations and finishing}
With all that information, we are able to replicate the same method used by Box in \cite{Box19} and to show that there are no other quadratic points on $X_0(91)$ apart from the known ones and the pullbacks of rational points on $X_0(91)^+$. All the exceptional (non-pullback) points on $X_0(91)$ are the four cusps and a pair of conjugate CM points. The remaining, non-exceptional quadratic points, are pullbacks of rational points on $X_0(91)^+$. Hence, $w_{91}$ acts the same way on them as Galois conjugation. If some non-exceptional point represents an $E$ with a rational $j$-invariant, we see that it is $91$-isogenous to itself since $j(E) = j(E^{\sigma})$ in that case. Therefore, a non-CM $E/\mathbb{Q}$ can't have a cyclic $91$-isogeny defined over a quadratic field. Also notice that all the rational points on $X_0(91)^+$ have been determined in \cite[Example 7.1.]{bal21}, so we are also able to get all the quadratic points on $X_0(91)$.

\subsection{Model and data for $X_0(91)$}
Model for $X_0(91)$:{\tiny
\begin{align*} 
&x_0^2 - 12x_1x_2 + 4x_1x_4 - 14x_2^2 + 12x_2x_3 + 24x_2x_4 - 14x_3^2 + 16x_3x_4 - 23x_4^2 - x_5^2 - 4x_6^2 = 0,\\
&x_0x_1 - 6x_1x_2 + 6x_1x_4 - 3x_2^2 + 2x_2x_3 + 7x_2x_4 - 5x_3^2 + 8x_3x_4 - 7x_4^2 - x_5x_6 - x_6^2 = 0,\\
&x_0x_2 - 2x_1x_2 + x_1x_4 - 3x_2^2 + 6x_2x_3 + 4x_2x_4 - 5x_3^2 + 4x_3x_4 - 3x_4^2 - x_6^2 = 0,\\
&x_0x_3 - x_1x_2 + x_1x_4 + 2x_2x_3 - x_2x_4 - x_3^2 + x_3x_4 + x_4^2 = 0,\\
&x_0x_4 - x_2^2 + 2x_2x_3 - x_3^2 + 2x_4^2 = 0,\\
&x_0x_6 - x_1x_5 + x_2x_5 + x_4x_6 = 0,\\
&x_1^2 - 2x_1x_2 - 3x_2^2 + 4x_2x_3 + 4x_2x_4 - 4x_3^2 + 4x_3x_4 - 4x_4^2 - x_6^2 = 0,\\
&x_1x_3 - x_1x_4 - x_2^2 + x_2x_3 + x_2x_4 - x_3x_4 = 0,\\
&x_1x_6 - x_2x_5 + x_3x_5 = 0,\\
&x_2x_6 - x_3x_5 + x_4x_5 - x_4x_6 = 0.\\
\end{align*}
}
Genus of $X_0(91)$: $7$.

Atkin-Lehner: $w_{91}(X_0 : X_1 : X_2 : X_3 : X_4 : X_5 : X_6) = (X_0 : X_1 : X_2 : X_3 : X_4 : -X_5 : -X_6)$.

Cusps:
$(1 : 0 : 0 : 0 : 0 : 1 : 0),
(-1 : 0 : 0 : 0 : 0 : 1 : 0),
(2 : 0 : -1 : -1 : -1 : 1 : 1),
(-2 : 0 : 1 : 1 : 1 : 1 : 1)$.

$C = X_0(91)^+$: hyperelliptic curve $y^2 = x^6 + 2x^5 - x^4 - 8x^3 - x^2 + 2x + 1$.

Group structure of $J_0(91)(\mathbb{Q})$: $J_0(91)(\mathbb{Q}) \simeq  \mathbb{Z} \oplus \mathbb{Z} \oplus \mathbb{Z}/2\mathbb{Z} \oplus \mathbb{Z}/168\mathbb{Z}$.

The only quadratic points on $X_0(91)$ are cusps (which are all defined over $\mathbb{Q}$), pullbacks of rational points on $X_0(91)^+$ and a pair of CM points $P, P^{\sigma}$ fixed by $w_{91}$, where:
\[
P = \Big(\frac{-8\alpha + 7}{5} : \frac{3\alpha - 7}{5} : \frac{-\alpha + 9}{5} : \alpha : 1 : 0 : 0\Big), \quad \alpha = \frac{17 + 5\sqrt{13}}{18}.
\]

\begin{acknowledgments}
The author gratefully acknowledges support
from the QuantiXLie Center of Excellence, a project co-financed by the Croatian Government and European Union through the
European Regional Development Fund - the Competitiveness and Cohesion Operational Programme (Grant KK.01.1.1.01.0004) and
by the Croatian Science Foundation under the project no. IP-2018-01-1313.
\end{acknowledgments}

\nocite{*}
\bibliographystyle{babplain-fl}
\bibliography{IsogeniesQuadratic}

\end{document}